\documentclass[12pt,draft]{article}
\usepackage{amsmath, amsfonts, amssymb, amsthm,epsfig}
\usepackage[obeyDraft,textsize=tiny]{todonotes}

\let\BFseries\bfseries\def\bfseries{\BFseries\mathversion{bold}} 

\newcommand{\eps}{\varepsilon}

\def\E{\mathbb{E}}

\theoremstyle{plain}
\newtheorem{thm}{Theorem}
\newtheorem{lem}[thm]{Lemma}

\theoremstyle{definition}

\newtheorem{remark}[thm]{Remark}

\renewenvironment{proof}[1][] {\smallskip \noindent {\bf Proof#1.} }{\hspace*{\fill}$\square$\medskip\par}

\def\e{\mbox{e}}
\def\P{{\bf {\mathbb{P}}}}

\def\N{{\bf {\mathbb{N}}}}
\def\R{{\bf {\mathbb{R}}}}
\def\d{\,\mbox{d}}

\def\E{\mathbb{E}}
\def\P{\mathbb{P}}
\def\PP{\mathbb{P}}
\def\EE{\mathbb{E}}
\begin{document}

\title{Persistence probabilities and a decorrelation inequality for the Rosenblatt process and Hermite processes}

\author{Frank Aurzada and Christian M\"onch}

\date{\today}

\maketitle

\begin{abstract}
We study persistence probabilities of Hermite processes. As a tool, we derive a general decorrelation inequality for the Rosenblatt process, which is reminiscent of Slepian's lemma for Gaussian processes or the FKG inequality and which may be of independent interest. This allows to compute the persistence exponent for the Rosenblatt process. For general Hermite processes, we derive upper and lower bounds for the persistence probabilites with the conjectured persistence exponent, but with non-matching boundaries.
\end{abstract}


\vspace{0.5cm}

\vspace{0.5cm}
\section{Introduction and results}
In this paper, we are interested in \emph{persistence probabilities} for stochastic processes, i.e., as $T\to\infty$, we study the quantity
$$F(T):=\P(\sup_{t\in [0,T]}Z_t\leq 1),$$ where $(Z_t)_{t\geq 0}$ denotes a stochastic process. The study of persistence probabilites was initiated in statistical physics and has received considerable attention there; an overview can be found in \cite{BMS13} and \cite{Maj1}. Persistence is conceived as a measure of how fast a physical system started in a disordered state returns to the equilibrium.

A survey about persistence from a mathematical point of view is given in \cite{AS}. More recently, research has been concentrated on persistence problems for non-Gaussian self-similar processes with complex dependence structures, see e.g.\ \cite{nG1,nG2} confirming \cite{MD,Maj2,Red1,ORS}. 

Our particular focus is on \emph{Hermite processes}, defined in terms of multiple Wiener-It\={o} integrals parametrised by an \emph{order} $m\in \N$ and a real number $H\in (\frac12,1)$: for $t\geq 0$, set
\begin{equation}\label{def:herm}
Y_t=Y^{(m,H)}_t := c_{H,m} \int^{'}_{\R^m} \Big(\int_{0}^t \prod_{i=1}^ m (s-x_i)_+^{\frac{H-1}{m}-\frac12}\textrm{ d}s \Big)\textrm{ d}W_{x_1}\dots\textrm{ d}W_{x_m},
\end{equation}
where $\int^{'}$ denotes integration excluding the diagonal, $(W_x)_{x\in\R}$ is standard Brownian motion, and $c_{H,m}$ is a normalising constant usually chosen to standardise the variance, but arbitrary in our context. We always work with a continuous modification.

Hermite processes play an important role in the study of long range dependence, since they arise as scaling limits of random walks with strongly correlated increments. We  sketch this connection in Section~\ref{sec:generalhermite}, cf.~\eqref{eq:hermconv}. A detailed account of the properties of $(Y_t)_{t\geq 0}$ is given e.g.\ in Chapter~3 of \cite{Tud}. For our present purpose it is sufficient to note that $(Y_t)_{t\geq 0}$ has stationary increments and is self-similar with index $H$.

It is conjectured that any $H$-self-similar process with stationary increments and finite exponential moments has \emph{persistence exponent} $1-H$, i.e.\ satisfies 
\begin{equation}\label{eq:goal}
F(T)= T^{-(1-H)+o(1)}, \quad \textrm{ as }T\to \infty.
\end{equation}

The Hermite process of order $m=1$ is fractional Brownian motion for which the exponent has been calculated in \cite{Molch} confirming a prediction from \cite{dingyang}; and refined estimates for the error function implicit in \eqref{eq:goal} are provided in \cite{A11,ab,agpp}. In contrast to fractional Brownian motion, the order $m=2$ Hermite process, i.e.\ the \emph{Rosenblatt process} $Z_t:=Y^{(2,H)}_t$, and all higher order Hermite processes are non-Gaussian. The analysis of $F(T)$ is thus a priori considerably more complicated. In particular, Slepian's inequality which is a main tool in proving \eqref{eq:goal} in the Gaussian setting cannot be applied.

The purpose of the present note is threefold:
\begin{itemize}
 \item We provide a new decorrelation inequality for the Rosenblatt processes (Theorem~\ref{thm:decorrelation}), which plays the role of Slepian's inequality in the present, non-Gaussian setting.
 \item We verify $F(T)=T^{-(1-H)+o(1)}$ for the Rosenblatt process (Theorem~\ref{thm:mainpersistence}).
 \item For general Hermite processes, we prove bounds for the persistence probability of the order $T^{-(1-H)+o(1)}$, however, unfortunately with non-matching boundaries (Theorem~\ref{thm:generalhermite}).
\end{itemize}

Our first result is a general decorrelation inequality for the Rosenblatt process. It says that one can decorrelate persistence events along a sample path of the Rosenblatt process.

\begin{thm} \label{thm:decorrelation}
 Let $(Z_t)_{t\geq 0}$ be a Rosenblatt process. Then for any $d\in \N$ and any $t_0<t_1<\ldots< t_d \in \R$ and $a_1,\ldots, a_d \in \R$ we have
$$
\P( \forall i=1,\ldots, d \sup_{t\in[t_{i-1},t_i)} Z_t - Z_{t_{i-1}} \leq a_i ) \geq \prod_{i=1}^d \P( \sup_{t\in[t_{i-1},t_i)} Z_t - Z_{t_{i-1}} \leq a_i  ).
$$
\end{thm}

We apply Theorem~\ref{thm:decorrelation} as a tool to calculate the persistence exponent of the Rosenblatt process, thus providing an instance in which \eqref{eq:goal} holds for a long-range dependent {\it non-Gaussian} process. More precisely, we prove the following:

\begin{thm} \label{thm:mainpersistence}
 Let $(Z_t)_{t\geq 0}$ be a Rosenblatt process with self-similarity index $H\in(\frac12,1)$. The following estimate for the persistence probability holds for some $c=c(H)>0$ and any sufficiently large $T$:
$$
T^{-(1-H)} (\log T)^{-c} \leq \P( \sup_{t\in[0,T]} Z_t \leq 1) \leq c T^{-(1-H)}.
$$
\end{thm}

The rest of this paper is structured as follows. In Section~\ref{sec:proofofthm1}, we prove Theorem~\ref{thm:decorrelation}, while in Section~\ref{sec:proofofthm2} we prove Theorem~\ref{thm:mainpersistence}. In Section~\ref{sec:generalhermite}, we obtain similar but weaker results for general Hermite processes.

\section{Proof of the decorrelation inequality}\label{sec:proofofthm1}
In this section we prove Theorem~\ref{thm:decorrelation}. We first discretise $(Z_t)_{t\geq 0}$, then the discretised version of the decorrelation inequality is a consequence of the Gaussian correlation inequality which we now recall:
\begin{lem}[Gaussian Correlation Inequality, {\cite{CGI,latalamatlak}}]\label{CGC}
Let $X\in\R^n$ be a centered Gaussian vector and $C_1,C_2\subset \R^n$ convex sets satisfying $C_i=\{-x: x\in C_i\}, i=1,2$. Then $$\P(X\in C_1\cap C_2)\geq \P(X\in C_1)\P(X\in C_2).$$
\end{lem}

\begin{proof}[ of Theorem~\ref{thm:decorrelation}] By self-similarity, we may consider $(Z_t)_{t\in[0,1]}$ only. Let $X_0,X_1,X_2,\dots$ be a stationary sequence of $\mathcal{N}(0,1)$ r.v.\ with covariance function \begin{equation}
	\label{eq:cov}
	\E(X_0 X_j)=(1+j^2)^{-\frac{1-H}{2}}.
	\end{equation}
Set furthermore $\sigma:=\sqrt{H^2-\frac{H}{2}}$, $h_2(x):=x^2-1$, for $x\in\R$ and
	\begin{equation}\label{def:sums}
	S_u:= \sum_{i=0}^{\lfloor u \rfloor} h_2(X_i),\quad u\in [0,\infty).
	\end{equation}
	It is well known, see e.g.\ \cite{TaqRP}, that the sequence of processes $(\frac{\sigma}{n^H}S_{nt})_{t\in[0,1]}$ converges in law to $(Z_t)_{t\in[0,1]}$. By continuity, Theorem~\ref{thm:decorrelation} thus follows, once we establish that, for any choice of $n_0<n_1<\dots< n_d\in \N$ and $a_1,\dots,a_d\in \R$,
 \begin{equation}\label{eq:cor1}
 \begin{aligned}
 &\PP\big(\forall i=1,\dots,d \max_{k=n_{i-1}+1,\dots,n_{i}} S_k-S_{n_{i-1}}\leq a_i\big)\\
 &\geq \prod_{i=1}^d \PP\big(\max_{k=n_{i-1}+1,\dots,n_{i}} S_k-S_{n_{i-1}}\leq a_i\big).
 \end{aligned}
 \end{equation}
Inequality \eqref{eq:cor1} may be rewritten in the following way:
\[\PP \Big((X_0,\dots,X_{n_d})\in \bigcap_{i=1}^{d} K_i\Big)\geq \prod_{i=1}^d \PP\big((X_0,\dots,X_{n_d})\in K_i\big),\]
where \[K_i=\Big\{(x_0,\dots,x_{n_d})\in \R^{n_d+1}:\;\sum_{j= n_{i-1}+1}^{k}x_j^2-1\leq a_i \;\forall k=n_{i-1}+1,\dots,n_i\Big\}.\] Hence \eqref{eq:cor1} follows immediately from Lemma~\ref{CGC} upon confirming all $K_i$ to be symmetric and convex. Since $h_2$ is an even function, symmetry of the $K_i$ follows. For convexity let \[M_{k,l}(a)=\{x\in \R^{n_d+1}:\; h_2(x_k)+\dots+h_2(x_l)\leq a\},\;0\leq k\leq l\leq n_d, a\in\R.\] $M_{l,k}(a)$ is the Cartesian product of $\R^{n_d+k-l-1}$ and the Euclidean ball of radius $\sqrt{(a+l-k+1)}$ in $\R^{l-k+1}$, thus convex. Each $K_i$ may be written as an intersection of finitely many $M_{k,l}(a)$ and is therefore convex, which concludes the argument.
\end{proof}

\begin{remark}
Although the Rosenblatt process is non-Gaussian, the Gaussian Correlation Inequality \cite{CGI,latalamatlak} is crucial in establishing Theorem~\ref{thm:decorrelation}.
\end{remark}

\section{Proof of the persistence result} \label{sec:proofofthm2}
In this section, we prove Theorem~\ref{thm:mainpersistence}. As in Section~\ref{sec:proofofthm1}, we first discretise the process. We then provide the upper and lower bound in Theorem~\ref{thm:mainpersistence} separately for the discrete setting.\

\begin{proof}[ of Theorem~\ref{thm:mainpersistence}]
\noindent
{\it Step 1: Reduction to the discrete time case.}

Note that trivially
\begin{equation}\label{eqn:upperboundboundary1}
\P( \sup_{t\in[0,T]} Z_t \leq 1) \leq \P( \max_{k=1,\ldots,\lceil T\rceil} Z_k \leq 1),
\end{equation}
so that in order to establish the upper bound in the theorem we only need an upper bound for the probability of the discrete event on the right hand side.

For the lower bound, note that by self-similarity, for any $a>0$,
\begin{eqnarray}
&&\P( \sup_{t\in[0,T]} Z_t \leq 1) \notag \\
&\geq &
\P(  \max_{k=1,\ldots,\lceil T (\log T)^a \rceil} Z_{\frac{k}{(\log T)^a}} < 0, \max_{k=1,\ldots,\lceil T (\log T)^a \rceil} \sup_{t\in(\frac{k-1}{(\log T)^a},\frac{k}{(\log T)^a}]} Z_t - Z_{\frac{k-1}{(\log T)^a}} \leq 1  ) 
\notag \\
&\geq &
\P(  \max_{k=1,\ldots,\lceil T (\log T)^a \rceil} Z_{\frac{k}{(\log T)^a}} < 0) 
\notag \\
&& \qquad - \P(\max_{k=1,\ldots,\lceil T (\log T)^a \rceil} \sup_{t\in(\frac{k-1}{(\log T)^a},\frac{k}{(\log T)^a}]} Z_t - Z_{\frac{k-1}{(\log T)^a}} > 1  ). \label{eqn:twotlot}
\end{eqnarray}
If we find a lower bound for the discrete setting with boundary $0$
\begin{equation} \label{eqn:lowerbdis0}
\P(  \max_{k=1,\ldots,\lceil T (\log T)^a \rceil} Z_{\frac{k}{(\log T)^a}} < 0) = \P(  \max_{k=1,\ldots,\lceil T (\log T)^a \rceil} Z_{k} < 0)
\end{equation}
and if we are able to show that the second term in \eqref{eqn:twotlot} is of lower order we are done with the proof of Theorem~\ref{thm:mainpersistence}.

For the second part, note that
\begin{eqnarray*}
&&
\P(\max_{k=1,\ldots,\lceil T (\log T)^a \rceil} \sup_{t\in(\frac{k-1}{(\log T)^a},\frac{k}{(\log T)^a}]} Z_t - Z_{\frac{k-1}{(\log T)^a}} > 1  )
\\
&\leq & T (\log T)^a \P( \sup_{t\in(0,(\log T)^{-a}]} Z_t > 1  )
\\
&= & T (\log T)^a \P( \sup_{t\in(0,1]} Z_t > (\log T)^{a H}  ),
\end{eqnarray*}
which decays faster than polynomially for any $a>\frac1H$, by the fact that, for sufficiently large $u$,
\begin{equation}\label{eq:tailest1}
\P( \sup_{t\in(0,1]} Z_t > u ) \leq e^{- c u},
\end{equation}
 where $c$ is some constant only depending on $H$. A tail estimate for more general self-similar processes which implies \eqref{eq:tailest1} is given in Lemma~6.3 of\ \cite{MO86}. Therefore, it only remains to identify the polynomial order of decay of the right hand side term in \eqref{eqn:lowerbdis0}.

\medskip
\noindent
{\it Step 2: Upper bound for the discrete time case.}

First note that by self-similarity, for any $p\geq 1$, \begin{equation*} 
 n^{pH}\EE \big(\max_{k=1,\ldots,n} Z_{\frac kn}\big)^p \,=\, \EE \max_{k=1,\ldots,n} Z^p_k\, \leq \,\EE \sup_{0\leq t\leq 1} Z^p_{nt}\, = \, n^{pH}\EE \big(\sup_{0\leq t\leq 1} Z_t\big)^p,\end{equation*}
and hence by continuity of $Z$,
\begin{equation}\label{eq:moments}
\lim_{n\to\infty} \frac{1}{n^{pH}}\E \big(\max_{k=1,\ldots,n} Z_k\big)^p = \EE \big(\sup_{0\leq t\leq 1} Z_t\big)^p\in(0,\infty).
\end{equation}

Therefore, setting $p=1$, we get the upper bound from the first part of Theorem 5 of \cite{agpp}:
\begin{equation} \label{eqn:upperboundfrom5}
 \P( \max_{k=1,\ldots,n} Z_k \leq -1) \leq c n^{-(1-H)}.
\end{equation}
In order to transfer from the boundary $-1$ to the boundary $+1$ required in (\ref{eqn:upperboundboundary1}), we make use of our Theorem~\ref{thm:decorrelation}:
\begin{eqnarray*}
\P( \max_{k=1,\ldots,n} Z_k \leq -1)
&\geq& \P( Z_1 \leq -2, \max_{k=2,\ldots,n} Z_k - Z_1 \leq 1) 
\\
&\geq& \P( Z_1 \leq -2) \cdot \P( \max_{k=2,\ldots,n} Z_k - Z_1 \leq 1)
\\
&=& \P( Z_1 \leq -2) \cdot \P( \max_{k=1,\ldots,n-1} Z_k \leq 1).
\end{eqnarray*}
Therefore, the upper bound from (\ref{eqn:upperboundfrom5}) transfers into an upper bound for the boundary $+1$ and thus shows the upper bound in the theorem.

\medskip
\noindent
{\it Step 3: Lower bound for the discrete time case.}

Observe that choosing $p=2$ in \eqref{eq:moments} yields
$\left(\E \max_{k=1,\ldots,n} Z_k^2 \right)^{1/2} \sim c n^H$. Further, recall that on $(-\infty,0]$ the density of the Rosenblatt distribution is bounded by a multiple of the standard Gaussian density (see e.g.\ Corollary~4.4 of\ \cite{VT13}). Hence, for some $\eps>0$,
\begin{equation}\label{eq:laplbd}
\E e^{\eps (Z_1)_-^{2}} < \infty.
\end{equation}
Therefore, the second part of Theorem 5 of \cite{agpp} yields
\begin{equation}\label{eq:lowerbdfin}
\P( \max_{k=1,\ldots,n} Z_k < 0) \geq n^{-(1-H)} (\log n)^{-1/2},
\end{equation}
as required in (\ref{eqn:lowerbdis0}).
\end{proof}

\section{Results for general Hermite processes} \label{sec:generalhermite}
Finally, we deal with general Hermite processes. It turns out that the proof of the lower bound in Theorem~\ref{thm:mainpersistence} does not require Theorem~\ref{thm:decorrelation} and works for general Hermite processes. Upper bounds on the persistence probability with different boundaries can be inferred from the recent article \cite{agpp}.

\begin{thm} \label{thm:generalhermite}
Let $(Y_t)_{t\geq 0}$ be an Hermite process of order $m\in \N$ and self-similarity index $H\in(\frac12,1)$. Then, for some constant $c=c(m,H)>0$ and for all sufficiently large $T$, we have
$$
\P( \sup_{t\in[0,T]} Y_t \leq 1) \geq T^{-(1-H)} (\log T)^{-c}
$$
and
$$
\P( \sup_{t\in[1,T]} Y_t \leq -1) \leq c T^{-(1-H)}.
$$
\end{thm}
\begin{proof}
Let $(Y_t)_{t\geq 0}$ be an $H$-self-similiar Hermite process of order $m$. Firstly, by Lemma~6.3 in\ \cite{MO86}, we may replace \eqref{eq:tailest1} by \begin{equation}\label{eq:tailest2}
\P( \sup_{t\in(0,1]} |Y_t| > u ) \leq e^{- c u^{2/m}},
\end{equation}
for some constant $c$ depending on $m$ and $H$ only and all sufficiently large $u$. Secondly, the conclusion of Theorem 5 in \cite{agpp} remains valid, if \eqref{eq:laplbd} is replaced by \begin{equation} \label{eqn:momentsofhermite}
\E e^{\eps (Y_1)_-^{\beta}} < \infty,
\end{equation}
for some $\beta>0$. In this case the lower bound \eqref{eq:lowerbdfin} becomes
\begin{equation}\label{eq:weakbound}
\P( \max_{k=1,\ldots,n} Y_k < 0) \geq n^{-(1-H)} (\log n)^{-1/\beta}.
\end{equation}
The Hermite process satisfies (\ref{eqn:momentsofhermite}) for $\beta\in (0,\frac 2m)$, which suffices to establish the lower bound in Theorem~\ref{thm:mainpersistence} for {any} Hermite process. However, \eqref{eq:weakbound} yields a weaker estimate than the one obtained for the Rosenblatt process since no sub-Gaussianity of the lower tail is used.

The upper bound in the theorem follows directly from Theorem 5 in\ \cite{agpp} because
$$
\P( \sup_{t\in[1,T]} Z_t \leq -1) \leq \P( \max_{k=1,\ldots,\lceil T\rceil} Z_k \leq -1).
$$
\end{proof}
It is precisely the use of Theorem~\ref{thm:decorrelation} which allows us to switch from boundary $-1$ to boundary $+1$ in the proof of Theorem~\ref{thm:mainpersistence}. One may wonder where the proof of Theorem~\ref{thm:decorrelation} breaks down for general Hermite processes. We are going to sketch this now in the remaining part of the paper.

The function $h_2$ appearing in the proof of Theorem~\ref{thm:decorrelation} is in fact the second \emph{Hermite polynomial}. Recall that, for $m\in\N\cup\{0\}$, the $m$-th Hermite polynomial is given by $$h_m(x):=(-1)^m e^{x^2/2} \frac{_{\textrm{d}^m}}{^{\textrm{d} x^m}} e^{-x^2/2}, \quad x\in\R.$$ Replacing $h_2$ by another Hermite polynomial $h_m, m\geq 1$, in \eqref{def:sums} yields, after suitable normalisation, the process $(Y^{(m,H)}_t)_{t\in[0,1]}$ as scaling limit of the sums $(S_{nt})_{t\in[0,1]}, n\in\N$.

To be more precise, the following invariance principle holds for a large class of long range dependent random walks, see\ \cite{Dob79,TaqRP}: Let $\Phi$ denote the standard normal distribution and set \[L^2:=L^2(\R,\Phi)=\Big\{f:\;\int f^2\d \Phi<\infty\Big\}.\]
$L^2$ is a Hilbert space with inner product $(f,g)_{\Phi}:=\sqrt{2\pi}\int fg\d \Phi$ for $f,g\in L^2$. Since the Hermite polynomials $(h_m)_{m=0}^{\infty}$ form an orthonormal basis of ${L}^2$, every $f\in L^2$ has a unique expansion
\begin{equation}\label{eq:hermex0}
f=\sum_{i=0}^{\infty}c_i h_i,
\end{equation}
in terms of Hermite polynomials. The index of the minimal nonzero coefficient appearing in \eqref{eq:hermex0} is called \emph{Hermite rank} of $f$. Let now $(X_i)_{i=0}^{\infty}$ be \emph{any} sequence of normal r.v.\ with the same polynomial decay of covariances as in \eqref{eq:cov}, $f\in L^2$ be \emph{any} function of {Hermite rank} $m\in\N$, and define \begin{equation}\label{def:sums2}
S_u:= \sum_{i=0}^{\lfloor u \rfloor} f(X_i),\quad u\in [0,\infty).
\end{equation} Then there is a constant $\sigma>0$ such that,
\begin{equation}
\label{eq:hermconv}
\big(\frac{_\sigma}{^{n^H}}S_{nt}\big)_{t\in[0,1]}\overset{}{\rightarrow}\big(Y^{(m,H)}_t\big)_{t\in[0,1]}, \quad  \textrm{in distribution as } n\to\infty.
\end{equation}

The proof of Theorem~\ref{thm:decorrelation} and also the upper bound in Theorem~\ref{thm:mainpersistence} could immediately be transfered to all Hermite processes, if for every $m\geq 3$, there existed a convex function $f$ of Hermite rank $m$. Unfortunately, as the following calculation shows, this approach is doomed to failure.
\begin{lem} \label{lem:lemmahermiterank}
	Let $f\in{L}^2$ be a nonzero convex function. Then $f$ has Hermite rank at most $2$.
\end{lem}
\begin{proof}
	Let $f\in L^2$ be convex and expand it into a Hermite series according to \eqref{eq:hermex0}, where $c_j=(f,h_j)_{\Phi}$, $j=0,1,\dots$. By convexity, $f$ is a continuous function and furthermore has a distributional second derivative $\nu\geq 0$ which is a Radon measure on $(\R,\mathcal{B}(\R))$.
	
	Assume now for contradiction that the Hermite rank of $f\neq 0$ equals $m\geq 3$, i.e.\ $c_0=c_1=c_2=0$. Using $\frac{\textrm{d}^2}{\textrm{d} u^2}\e^{-u^2/2}=h_2(u)\e^{-u^2/2}$ and integration by parts, we have
	\[c_2=(f,h_2)_{\Phi}=\int f(u)h_2(u)\e^{-\frac{u^2}{2}}\d u=\int f(u)\frac{_{\textrm{d}^2}}{^{\textrm{d} u^2}}\e^{-\frac{u^2}{2}}\d u=\int \e^{-\frac{u^2}{2}}\nu(\textrm{d} u).\] Since $\e^{-u^2/2}>0$ for all $u\in\R$, we obtain $c_2=0$ if and only if $\nu=0$. Thus $f$ must be affine, i.e.\ $f\in\textrm{span}\{h_0,h_1\}$ and $c_0=c_1=0$ means that $f$ vanishes. We conclude that our assumption cannot be satisfied.
\end{proof}



\end{document}